\documentclass[12pt, reqno]{amsart}
\usepackage{amsmath, amsfonts, amsbsy, amsthm, amscd, graphicx}
\usepackage{amssymb, latexsym, color}
\usepackage{bm}
\usepackage[svgnames,psnames]{xcolor}
\usepackage[colorlinks, citecolor=Green,linkcolor=FireBrick,linktocpage,unicode]{hyperref}

\theoremstyle{plain}
\numberwithin{equation}{section}
\newtheorem{theo}{Theorem}[section]

\newtheorem{lemm}[theo]{Lemma}
\newtheorem{rema}[theo]{Remark}
\newtheorem{corr}[theo]{Corollary}

\newtheorem{prop}[theo]{Proposition}

\begin{document}

\title{Convex Hulls of Grassmannians and Combinatorics of Symmetric Hypermatrices}
\author{Kazumasa Narita}
\thanks{Graduate School of Mathematics, Nagoya University, Furo-cho, Chikusa-ku, Nagoya 464-8602, Japan, m19032e@math.nagoya-u.ac.jp}
\date{}

\maketitle

\begin{abstract}
It is known that the complex Grassmannian of $k$-dimensional subspaces can be identified with the set of projection matrices of rank $k$. It is also classically known that the convex hull of this set is the set of Hermitian matrices with eigenvalues between $0$ and $1$ and summing to $k$. We give a new proof of this fact. We also give an existence theorem for a certain combinatorial class of hypermatrices by a similar argument. This existence theorem can be rewritten into an existence theorem for a uniform weighted hypergraph with given weighted degree sequence.
\end{abstract}

\section{Introduction}

 The convex hull of the set of projection matrices of constant rank is classically known. In this paper, we give a new approach to identify this convex hull and study combinatorics of symmetric hypermatrices. These subjects appear to be unrelated, but the proofs of both results are similar.

Tai \cite{Tai} showed that projective spaces can be embedded isometrically into certain Euclidean spaces of matrices. More explicitly, the complex projective space $\mathbf{C}P^{n}$, for example, can be embedded in a space of Hermitian matrices as follows:
\begin{equation*}
\mathbf{C}P^{n} \stackrel{\cong}{\longrightarrow} \{ A \in HM(n+1) \mid A^{2} = A, \mathrm{tr}A=1 \}, [z] \mapsto \frac{1}{z^{*}z} zz^{*}
\end{equation*}
Here $z$ is a column vector in $\mathbf{C}^{n+1}$. This embedding is well-defined and sends $[z]$, which can be identified with a complex line in $\mathbf{C}^{n+1}$, to the projection matrix onto the complex line $[z]$. Later, Dimitri\'{c} \cite{Dimitric} showed that Tai's embedding can be generalized to Grassmannians by identifying points in Grassmannians and the corresponding projection matrices. That is, the complex Grassmannian of $k$-dimensional subspaces $G_{k}(\mathbf{C}^{n})$ can be embedded into $HM(n)$ as
\begin{equation}
G_{k}(\mathbf{C}^{n}) \cong \{ A \in HM(n) \mid A^{2} = A, \: \mathrm{tr}A = k \}.
\end{equation}
If a Hermitian matrix $A$ is positive semidefinite, we write $A \succeq 0$. Fillmore and Williams \cite{FW} proved the following theorem:

\begin{theo}\emph{(\cite{FW})} 
\label{fw}
The convex hull of the embedded Grassmannian $G_{k}(\mathbf{C}^{n})$ in $HM(n)$ is
\begin{equation*}
\Omega := \{ A \in HM(n) \mid A \succeq 0,\:  I_{n}-A \succeq 0, \: \mathrm{tr} A = k \}
\end{equation*}
and the set of extreme points of $\Omega$ is $G_{k}(\mathbf{C}^{n})$. 
\end{theo}
They used this fact to study $k$-numerical ranges of matrices, the notion introduced by Halmos \cite[p.\,111]{Halmos}. Overton and Womersley \cite{OW} showed that Theorem \ref{fw} implies Fan's theorem \cite{Fan}, which is an extension of Rayleigh's principle and gives a characterization of the sum of the largest $k$ eigenvalues of a Hermitian matrix.

As Overton and Womersley \cite{OW} pointed out, the conventional proof of Fan's theorem relies heavily on the following classical theorem (see \cite[p.\,785]{MO}, for example):

\begin{theo}
\label{birkhoff}
Let $\Omega_{1}$ be the set of permutation matrices of order $n$ and $\Omega_{2}$ be the set of doubly stochastic matrices of order $n$. Then $\Omega_{2}$ is the convex hull of $\Omega_{1}$ and $\Omega_{1}$ is the set of extreme points of $\Omega_{2}$.
\end{theo} 
 

In this paper, we give another proof of Theorem \ref{fw} by using a combinatorial theorem by Rado \cite{Rado}. This approach should be of interest since it provides a new path deducing Fan's theorem from Theorem \ref{birkhoff} if one notices that Rado's theorem can be proved by using Theorem \ref{birkhoff} (see \cite[p.\,60]{Zhan2}, for example).

In the second part of the paper, we study combinatorics of symmetric hypermatrices and weighted uniform hypergraphs. A $d$-uniform hypergraph is a hypergraph each of whose edge has exactly $d$ vertices. A $d$-uniform hypergraph is called \textit{simple} if it has no repeated edges. A finite sequence of nonnegative integers is called \textit{$d$-graphic} if it can be realized as the degree sequence of some simple $d$-uniform hypergraph. In particular, a simple $2$-uniform hypergraph is nothing but a simple graph. Havel \cite{Havel} and Hakimi \cite{Hakimi} independently gave a characterization of a $2$-graphic sequence. Today there are many other characterizations including the classic Erd\H{o}s--Gallai theorem \cite{EG} (see \cite{HS}). Dewdney \cite{Dewdney} generalized the Havel--Hakimi algorithm and characterized a $d$-graphic sequence as follows:
\begin{theo}\emph{(\cite{Dewdney})}
\label{Dewdney}
Let $d, n \in \mathbf{N}$ with $d \geq 2$ and $n \geq d$. Let $D= (d_{1}, \ldots, d_{n})$ be a nonincreasing sequence of nonnegative integers. Then $D$ is $d$-graphic if and only if there exists a nonincreasing sequence $D' = (d'_{2}, \ldots, d'_{n})$ of nonnegative integers satisfying all of the following conditions:
\begin{enumerate}
\item $D'$ is $(d-1)$-graphic. 
\item $\sum_{i=2}^{n} d'_{i} = (d-1)d_{1}$.
\item $D'' := (d_{2}-d'_{2}, \ldots, d_{n}-d'_{n}, 0)$ is $d$-graphic.
\end{enumerate}
\end{theo}
In contrast to the case $d=2$, this is the the only known characterization of a $d$-graphic sequence for $d \geq 3$.  It is known that this characterization does not yield an efficient algorithm, so conditions for a sequence $D$ to be $d$-graphic have been investigated for a long time (see \cite{AAS}, \cite{Behrens} and \cite{Frosini}, for example).

On the other hand, Hakimi \cite{Hakimi} also characterized when a nonincreasing sequence of nonnegative real numbers is realized as the weighted degree sequence of some weighted graph. This chracterization can be interpreted as an existence theorem for a symmetric nonnegative matrix with given row sums and with only 0's on the main diagonal (see Section $4$ of this paper and also \cite[Section 2.3, 7.2]{Brualdi}). In this paper, we generalize Hakimi's theorem to symmetric hypermatrices by a similar argument to that for determining the convex hulls of Grassmannians. The notion of alternating sign matrices, one of combinatorial matrix classes, was generalized to hypermatrices by Brualdi and Dahl \cite{BD} recently. Still few results are known about combinatorics of hypermatrices. Hence the generalization of Hakimi's theorem to symmetric hypermatrices itself is of interest. In addition, the theorem can be rewritten into the following:
\begin{theo}
\label{MainTheorem}
Let $d, n \in \mathbf{N}$ with $d \geq 2$ and $n \geq d$. Let $D = (d_{1}, \ldots, d_{n})$ be a nonincreasing sequence of nonnegative numbers. Then there exists a $d$-uniform weighted hypergraph with the weighted degree sequence $D$ if and only if the inequality
\begin{equation}
\label{condition2}
dd_{1} \leq \sum_{i=1}^{n}d_{i}
\end{equation}
holds. Moreover, if this holds, then there exists a $d$-uniform weighted hypergraph with the weighted degree sequence $D$ that has at most $n$ edges with positive weights.
\end{theo} 
Letting $d=2$ recovers the aforementioned theorem due to Hakimi. (See Section 4 for the precise definitions of a $d$-uniform weighted hypergraph and a weighted degree sequence.)  At the end of this paper, we clarify the relationship between Theorem $\ref{Dewdney}$ and Theorem $\ref{MainTheorem}$.

\section{Majorization}
In this section, we set up tools from combinatorics that we use later. For $x = (x_{1}, \ldots ,x_{n}) \in \mathbf{R}^{n}$, we rearrange the components in decreasing order and label the new subscripts $x = (x_{[1]}, \ldots ,x_{[n]})$ so that we have $x_{[1]} \geq  \cdots  \geq x_{[n]}.$ If $x, y \in \mathbf{R}^{n}$ satisfy
\begin{equation*}
\sum_{i=1}^{k} x_{[i]} \leq \sum_{i=1}^{k} y_{[i]} \quad \mbox{for all} \quad 1 \leq k \leq n-1
\end{equation*}
and
\begin{equation*}
\sum_{i=1}^{n} x_{i} = \sum_{i=1}^{n} y_{i},
\end{equation*}
then we say that $x$ is \textit{majorized} by $y$, which is denoted by $x \preceq y$. If $\tilde{x} \in \mathbf{R}^{n}$ can be obtained by a rearrangement of components of  $x \in \mathbf{R}^{n}$, then $\tilde{x}$ is called a  \textit{permutation} of $x$ (see \cite[Chaper 1]{MO} or \cite[Section 3.2]{Zhan2}, for example).

Rado \cite{Rado} showed that for $x, y \in \mathbf{R}^{n}$, $x$ is majorized by $y$ if and only if $x$ is a convex combination of permutations of $y$. Although $y \in \mathbf{R}^{n}$ has $n!$ permutations in general, Zhan \cite{Zhan} pointed out that the number of permutations in the convex combination can be reduced less than or equal to $n$ by using Carath\'{e}odory's theorem. That is, we have the following result, which is a crucial tool in this paper:

\begin{theo}
\label{Rado}
Let $x, y \in \mathbf{R}^{n}$. Then $x$ is majorized by $y$ if and only if $x$ is a convex combination of at most $n$ permutations of $y$.
\end{theo}

\section{Convex Hulls of Grassmannians}

Let $\textbf{F} = \mathbf{R}$ or $\mathbf{C}$. The Grassmannian of $k$-dimensional subspaces of $\mathbf{F}^{n}$ is denoted by $G_{k}(\mathbf{F}^{n})$. We denote by $FM(n)$ the space of real symmetric matrices of order $n$ if $\mathbf{F} = \mathbf{R}$, or the space of Hermitian matrices of order $n$ if $\mathbf{F} = \mathbf{C}$. If $A \in FM(n)$ is positive semidefinite, we write $A \succeq 0$. It is known that the Grassmannian $G_{k}(\mathbf{F}^{n})$ can be identified with set of projectors in $FM(n)$:
\begin{equation}
\label{identification}
G_{k}(\mathbf{F}^{n}) \cong \{ A \in FM(n) \mid A^{2} = A, \: \mathrm{tr}A = k \}.
\end{equation}
This embedding is isometric with respect to the canonical metrics on Grassmannian and $FM(n)$ (see \cite{Dimitric} for details). Hereinafter, we identify the both sides of the equation $(\ref{identification})$. We set
\begin{equation}
\mathcal{H}_{k}(\mathbf{F}^{n}) := \{ A \in FM(n) \mid A \succeq 0,\:  I_{n}-A \succeq 0, \: \mathrm{tr} A = k \}.
\end{equation}
One can immediately see that $\mathcal{H}_{k}(\mathbf{F}^{n})$ is a convex subset of $FM(n)$. A point of a convex body $C \subset \mathbf{R}^{N}$ is called \textit{extreme} if it is not an interior point of any line segments in $C$. The set of extreme points of $C$ is denoted by $\mathrm{ext}C$. Fillmore and Williams \cite{FW} proved the following proposition: 
\begin{prop}
The Grassmannian $G_{k}(\mathbf{F}^{n})$ is the set of extreme points of $\mathcal{H}_{k}(\mathbf{F}^{n})$, that is, $G_{k}(\mathbf{F}^{n}) = \mathrm{ext}\mathcal{H}_{k}(\mathbf{F}^{n})$.
\end{prop}
The classical Krein--Milman theorem says that the convex hull of the extreme points of a convex set coincides with the convex set itself. Hence by the theorem, the above proposition implies the following theorem:
\begin{theo}\emph{(\cite{FW})}
\label{convex-hull}
The convex hull of $G_{k}(\mathbf{F}^{n})$ in $FM(n)$ is $\mathcal{H}_{k}(\mathbf{F}^{n})$.
\end{theo}

We give another proof of this theorem by using Theorem \ref{Rado}.

\begin{proof}
Each element $A \in G_{k}(\mathbf{F}^{n})$ satisfies $A^{2}=A$ and $\mathrm{tr}A = k$. This implies that the eigenvalues of $A$ are $1$ and $0$ with multiplicity $k$ and $n-k$ respectively. Hence it follows that $G_{k}(\mathbf{F}^{n}) \subset \mathcal{H}_{k}(\mathbf{F}^{n}).$ Since $\mathcal{H}_{k}(\mathbf{F}^{n})$ is convex, it suffices to show that each element $X \in \mathcal{H}_{k}(\mathbf{F}^{n})$ can be expressed as a convex combination of elements of $G_{k}(\mathbf{F}^{n}).$ For any $X \in \mathcal{H}_{k}(\mathbf{F}^{n})$, there exists $U$ such that
\begin{equation}
\label{r}
UXU^{-1} = \mathrm{diag}(r_{1}, \ldots, r_{n}), \quad 1 \geq r_{1} \geq r_{2} \geq \cdots \geq r_{n} \geq 0, \quad \sum_{i=1}^{n}r_{i} = k,
\end{equation}
where $U$ is an orthogonal matrix if $\mathbf{F} = \mathbf{R}$, or a unitary matrix if $\mathbf{F} = \mathbf{C}.$ For any $k$-tuple $(i_{1}, \ldots ,i_{k})$ with $1 \leq i_{1} < \cdots < i_{k} \leq n$, we denote by $A_{ i_{1} \cdots i_{k} }$ the diagonal matrix that has units in the $i_{j}$-th diagonal entries $(j=1, 2, \ldots ,k)$ and zeros in the others. Clearly, we have $A_{ i_{1}, \cdots i_{k} } \in G_{k}(\mathbf{F}^{n})$. Let $r:= (r_{1}, \ldots , r_{n}) \in \mathbf{R}^{n}$. Consider 
\begin{equation*}
x := (1, \ldots ,1, 0, \ldots 0) \in \mathbf{R}^{n},
\end{equation*}
where the number of $1$'s is $k$ and that of $0$'s is $n-k$. Then (\ref{r}) implies $r \preceq x$. Hence Theorem \ref{Rado} implies that $r$ can be written as a convex combination of permutations of $x$. Thus it follows that there exist $\{c_{ i_{1}, \cdots i_{k} } \geq 0 \mid 1 \leq i_{1} < \cdots < i_{k} \leq n \}$ such that
\begin{equation*}
UXU^{-1} = \mathrm{diag}(r_{1}, \cdots, r_{n}) = \sum_{ 1 \leq i_{1} < \cdots < i_{k} \leq n } c_{ i_{1} \cdots i_{k} } A_{i_{1} \cdots i_{k} }
\end{equation*}
and 
\begin{equation*}
\sum_{ 1 \leq i_{1} < \cdots < i_{k} \leq n } c_{ i_{1} \cdots i_{k} } =1.
\end{equation*}
Since we have $U^{-1} A_{i_{1} \cdots i_{k} } U \in G_{k}(\mathbf{F}^{n})$, $X$ can be written as a convex combination of elements in $G_{k}(\mathbf{F}^{n})$. This completes the proof.
\end{proof}

\section{Symmetric Hypermatrices}

Let $\mathbf{N}$ be the set of positive integers. For $n \in \mathbf{N}$, we set $\langle n \rangle := \{ 1, 2, \ldots, n\}$. Let $d \in \mathbf{N}$. For $n_{1}, \ldots, n_{d}$, a (real) \textit{d-hypermatrix} is defined to be a function $A: \langle n_{1} \rangle \times \cdots \times \langle n_{d} \rangle \rightarrow \mathbf{R}$. A 2-hypermatrix is just an ordinary matrix. A value $A(i_{1}, \ldots, i_{d})$ is called an \textit{entry}. 
\textbf{ In this paper, we only consider the case where $n_{1} = \cdots = n_{d} = n$ }. By fixing one index $i_{k}$ to be $l$, one gets a $(d-1)$-dimensional subarray $\{A(i_{1}, \ldots, i_{k-1}, l, i_{k+1}, \ldots, i_{d}) \mid i_{1}, \ldots i_{k-1}, i_{k+1}, \ldots, i_{d} \in \langle n \rangle \}$. This $(d-1)$-dimensional subarray is called a \textit{slice} of the $d$-hypermatrix with respect to $i_{k} = l$. A slice itself can be regarded as a $(d-1)$-hypermatrix. We set
\begin{equation*}
I_{k} := \left\{ (k, i_{2}, \ldots, i_{d} ) \in \langle n \rangle^{d} \mid i_{2} < \cdots < i_{d}, \quad  i_{2}, \ldots ,i_{d} \in \langle n \rangle \setminus \{k\} \right\}
\end{equation*}
for each $k \in \langle n \rangle$.

A $d$-hypermatrix $A$ is called \textit{nonnegative} if each entry of $A$ is nonnegative. Let $\mathfrak{S}_{d}$ be the symmetric group of degree $d$. A $d$-hypermatrix $A$ is called \textit{symmetric} if the equation
\begin{equation*}
A(i_{\pi(1)}, \ldots, i_{\pi(d)}) = A(i_{1}, \ldots, i_{d})
\end{equation*}
holds for any $(i_{1}, \ldots, i_{d}) \in \langle n \rangle^{d}$ and any $\pi \in \mathfrak{S}_{d}$. A $d$-hypermatrix $A$ is defined to be \textit{strongly hollow} if every entry $A(i_{1}, \ldots, i_{d})$ that has two identical indices $i_{k} = i_{l}$ $(k \neq l)$ is zero. A strongly hollow 2-hypermatrix is just an ordinary hollow matrix, i.e. a square matrix whose diagonal entries are all equal to zero.

Consider $R = (r_{1}, \ldots, r_{n}) \in \mathbf{R}^{n}$ with $r_{1} \geq \cdots \geq r_{n} \geq 0$. We denote by $\mathcal{N}_{0}(d; R)$ the set of strongly hollow symmetric nonnegative $d$-hypermatrices whose sum of all the entries in each slice with respect to $i_{1} = j$ $(j \in \langle n \rangle )$ is equal to $r_{j}$. Note that $i_{1}$ in the above definition can be replaced by any fixed $i_{k}$ $(k \in \langle n \rangle)$ since we consider a symmetric hypermatrix. A strongly hollow symmetric nonnegative $d$-hypermatrix $A$ belongs to  $\mathcal{N}_{0}(d; R)$ if and only if the equation
\begin{equation}
\label{slice-sum}
(d-1)! \sum_{\mathbf{i} \in I_{k} } A(\mathbf{i}) = r_{k}
\end{equation}
holds for each $k \in \langle n \rangle$.

The set $\mathcal{N}_{0}(d; R)$ might be empty. For example, $\mathcal{N}_{0}(d; R)$ is obviously empty when $n=1$ and $R = (1)$. One can easily verify the emptiness also when $n=2$ and $R=(2,1)$. Hence it is natural to look for the condition that the set $\mathcal{N}_{0}(d; R)$ is nonempty. Using a similar argument to that in the proof of Theorem \ref{convex-hull}, one can prove the following theorem.

\begin{theo}
\label{Hakimi}
Let $d, n \in \mathbf{N}$ with $d \geq 2$ and $n \geq d$. Let $R = (r_{1}, \ldots, r_{n})$ be a vector with $r_{1} \geq \cdots \geq r_{n} \geq 0$. Then the set $\mathcal{N}_{0}(d; R)$ is nonempty if and only if the inequality
\begin{equation}
\label{condition}
dr_{1} \leq \sum_{i=1}^{n}r_{i}
\end{equation}
holds. Moreover, if this holds, then there exists a $d$-hypermatrix in $\mathcal{N}_{0}(d; R)$ the number of whose nonzero entries is at most $n \cdot d!$. 
\end{theo} 

For $d=2$, the theorem was proved by Hakimi \cite{Hakimi} in a graph-theoretic way.

\begin{proof}
We assume $A \in \mathcal{N}_{0}(d; R).$  Then, the equation (\ref{slice-sum}) holds for each $k \in \langle n \rangle$. Fix $2 \leq i_{2} <i_{3} < \cdots <i_{d} \leq n$. The symmetry of $A$ implies 
\begin{equation}
\label{1-slice}
\begin{split}
&(d-1)\cdot A(1, i_{2}, \ldots, i_{d}) \\
&= A(i_{2}, 1, i_{3}, \ldots, i_{d}) + A(i_{3}, 1, i_{2}, i_{4}\ldots, i_{d}) + \cdots +A(i_{d}, 1, i_{2}, \cdots, i_{d-1}). \\
\end{split}
\end{equation}
Each of the ordered $d$-tuples $(i_{2}, 1, i_{3}, \ldots, i_{d}), (i_{3}, 1, i_{2}, i_{4}, \ldots, i_{d}), \ldots, (i_{d}, 1, i_{2}, \ldots, i_{d-1})$ belongs to one of $I_{2}, \ldots, I_{n}$. Since only the elements $\{ A(i_{1}, \ldots, i_{d}) \mid (i_{1}, \ldots, i_{d}) \in \cup_{j=2}^{n} I_{j}, \:  i_{l} = 1 \: \mbox{for some} \: l \}$ can appear on the right-hand side of (\ref{1-slice}), we have
\begin{equation*}
0 \leq r_{2}+ \cdots +r_{n} - (d-1)r_{1} ,
\end{equation*}
and so the inequality (\ref{condition}).

We show the converse. Set 
\begin{equation*}
S = \frac{1}{d}\sum_{i=1}^{n} r_{i}.
\end{equation*}
We know $r_{1} \leq S$ by the assumption. Consider a vector
\begin{equation*}
R' := \left( \frac{r_{1}}{S}, \ldots , \frac{r_{n}}{S} \right). 
\end{equation*}
For $1 \leq i_{1} <i_{2}< \cdots <i_{d} \leq n$, we denote by $v_{i_{1}i_{2}\cdots i_{d}}$ the vector in $\mathbf{R}^{n}$ whose $i_{1}, i_{2}, \ldots, i_{d}$ entries are equal to $1$ and the other entries are equal to $0$. Now we have
\begin{equation*}
R' \preceq v_{12\cdots d} = (1, \ldots 1, 0, \ldots ,0),
\end{equation*}
where the number of $1$'s is $d$ and that of $0$'s is $n-d$. Hence Theorem \ref{Rado} implies that there exist nonnegative constants $\{ c_{i_{1}i_{2}\cdots i_{d} } \geq 0 \mid 1 \leq i_{1} <i_{2}< \cdots <i_{d} \leq n \}$ such that 

\begin{equation*}
\sum_{1 \leq i_{1} <i_{2}< \cdots <i_{d} \leq n } c_{i_{1}i_{2}\cdots i_{d} }v_{i_{1}i_{2}\cdots i_{d}}  = R' \quad  \mbox{and} \quad \sum_{1 \leq i_{1} <i_{2}< \cdots <i_{d} \leq n } c_{i_{1}i_{2}\cdots i_{d} }=1.
\end{equation*}
 Then for arbitrary $k \in \langle n \rangle$, we obtain
\begin{equation*}
\sum_{l=1}^{d} \sum_{\substack{1 \leq i_{1} <i_{2}< \cdots i_{l-1} \\ < i_{l}= k  <i_{l+1}<\cdots <i_{d} \leq n } } c_{i_{1}i_{2}\cdots i_{d} } = \frac{r_{k}}{S}.
\end{equation*}
From the set $\{ c_{i_{1}i_{2}\cdots i_{d} } \geq 0 \mid 1 \leq i_{1} <i_{2}< \cdots <i_{d} \leq n \}$, one can naturally construct a strongly hollow symmetric $d$-hypermatrix $(\tilde{c}_{i_{1}i_{2}\cdots i_{d} })$ satisfying 
\begin{equation*}
\tilde{c}_{\pi(i_{1})\pi(i_{2})\cdots \pi( i_{d} ) } =  c_{i_{1}i_{2}\cdots i_{d} } \quad \mbox{for any} \quad \pi \in \mathfrak{S}_{d}.
\end{equation*} 
Since each $c_{i_{1}i_{2}\cdots i_{d} }$ is nonnegative, the hypermatrix $(\tilde{c}_{i_{1}i_{2}\cdots i_{d} })$ is nonnegative. Then for each $k \in \langle n \rangle$, we have
\begin{equation*}
\sum_{j_{1}, \cdots ,j_{d-1} \in \langle n \rangle \setminus {k} }  \tilde{c}_{kj_{1}\cdots j_{d-1}} = (d-1)!  \sum_{l=1}^{d} \sum_{\substack{1 \leq i_{1} <i_{2}< \cdots i_{l-1} \\ < i_{l}= k  <i_{l+1}<\cdots <i_{d} \leq n } } c_{i_{1}i_{2}\cdots i_{d} }= \frac{(d-1)!r_{k}}{S}.
\end{equation*}
Thus we conclude 
\begin{equation*}
\left( \frac{S}{(d-1)!}  \tilde{c}_{i_{1}i_{2}\cdots i_{d} } \right) \in \mathcal{N}_{0}(d; R).
\end{equation*}

The second half of the theorem follows from Theorem \ref{Rado} and the symmetry of $\left( \frac{S}{(d-1)!}  \tilde{c}_{i_{1}i_{2}\cdots i_{d} } \right)$.
 
\end{proof} 

In the following lemma, we show that if the set $\mathcal{N}_{0}(d; R)$ is nonempty for given $d \geq 2$ and $R \in \mathbf{R}^{n}$, it is a convex polytope, possibly a singleton.

\begin{lemm}
\label{polytope}
Let $d, n \in \mathbf{N}$ with $d \geq 2$ and $n \geq d$. Let $R = (r_{1}, \ldots, r_{n})$ be a vector with $r_{1} \geq \cdots \geq r_{n} \geq 0$. Assume that the set $\mathcal{N}_{0}(d; R)$ is nonempty. Then the set $\mathcal{N}_{0}(d; R)$ is a convex polytope for $n \geq d+2$, and a singleton for $n = d$ or $d+1$. 
\end{lemm}

\begin{proof}
First we consider the case where $n=d$. The assumption on $R$ and Theorem \ref{Hakimi} imply 
\begin{equation*}
dr_{1} \leq \sum_{i=1}^{d}r_{i} \leq dr_{1}
\end{equation*}
and so $r_{1} = \cdots =r_{d}$. Then the strongly hollow symmetric $d$-hypermatrix $A$ given by
\begin{equation}
\label{singleton2}
A( \pi(1), \ldots, \pi(d)) = A(1,2, \ldots, d)= \frac{r_{1}}{(d-1)!} \quad \mbox{for any} \quad \pi \in \mathcal{S}_{d}
\end{equation}
is the unique element in $\mathcal{N}_{0}(d; R)$.

Hereinafter, we consider the case where $n \geq d+1$. We rewrite the slice sum constraints (\ref{slice-sum}) as follows. Let $w$ be a vector in $\mathbf{R}^{\binom{n}{d}}$ obtained by aligning all the elements in $\{ A(i_{1}, \ldots, i_{d} ) \mid 1 \leq i_{1}< \cdots <i_{d} \leq n \}$ in an arbitrary order. For $1 \leq i_{1} <i_{2}< \cdots <i_{d} \leq n$, we denote by $v_{i_{1}i_{2}\cdots i_{d}}$ the vector in $\mathbf{R}^{n}$ whose $i_{1}, i_{2}, \ldots, i_{d}$ entries are equal to $1$ and the other entries are equal to $0$. We construct an $n \times \binom{n}{d}$ matrix $B$ by arranging the vectors $\{v_{i_{1}i_{2}\cdots i_{d}} \mid 1 \leq i_{1} <i_{2}< \cdots <i_{d} \leq n \}$ in the unique order so that if the $j$-th entry of $w$ is $A(i_{1}, \ldots, i_{d})$, then the $j$-th column of $B$ is $v_{i_{1}i_{2}\cdots i_{d}}$. Then the linear system (\ref{slice-sum}) is equivalent to
\begin{equation}
\label{slice-sum2}
Bw =
\frac{1}{(d-1)!} R^{T}. 
\end{equation}
Since each column of $B$ is a permutation of $v_{1 \cdots d}$ and we have $n \geq d+1$, one can easily check that the rank of $B$ is $n$. 

For $n=d+1$, the system (\ref{slice-sum2}) has $(d+1)$ equations and $(d+1)$ variables $\{A(1, \ldots, k-1, \check{k}, k+1, \ldots ,d+1) \mid 1 \leq k \leq d+1\}$, where $\check{k}$ denotes the omission of $k$. Therefore, the solutions $\{A(1, \ldots, k-1, \check{k}, k+1, \ldots ,d+1) \mid k \in \langle d+1 \rangle\}$ are uniquely determined. We find the solutions explicitly and show that they are nonnegative. Summing up all the equations in (\ref{slice-sum}), one obtains
\begin{equation}
\label{total}
\sum_{k=1}^{d+1}A(1, \ldots, k-1, \check{k}, k+1, \ldots ,d+1) = \frac{1}{d!}\sum_{i=1}^{d+1}r_{i}.
\end{equation}
For each $k \in \langle d+1 \rangle$, the left-hand side of the equation (\ref{slice-sum}) is the sum of all the variables except for $A(1, \ldots, k-1, \check{k}, k+1, \ldots ,d+1)$. Thus the difference between (\ref{total}) and (\ref{slice-sum}) gives
\begin{equation}
\label{singleton}
A(1, \ldots, k-1, \check{k}, k+1, \ldots ,d+1) = \frac{1}{d!}\left(\sum_{i=1}^{d+1}r_{i}\right) - \frac{r_{k}}{(d-1)!},
\end{equation}
which is nonnegative by the assumption (\ref{condition}). Hence the set $\mathcal{N}_{0}(d; R)$ is a singleton for $n=d+1$. 

Finally we consider the case where $n \geq d+2$. Then the system (\ref{slice-sum2}) has $n$ equations and $\binom{n}{d}$ variables. The assumption that $n \geq d+2 \geq 4$ implies 
\begin{equation*}
\binom{n}{d} \geq \frac{n(n-1)}{2} > n.
\end{equation*}
Hence the solution space of (\ref{slice-sum2}) is of dimension $\tbinom{n}{d} - n >0$. The assumption that a hypermatrix is nonnegative implies that the set $\mathcal{N}_{0}(d; R)$ is the intersection of half spaces and the solution space. Hence $\mathcal{N}_{0}(d; R)$ is a polytope of dimension $\tbinom{n}{d} - n$. Moreover, one can readily see that $\mathcal{N}_{0}(d; R)$ is a convex polytope. Thus the claim follows.
\end{proof}

In the above proof, we have described the unique element in $\mathcal{N}_{0}(d; R)$ in the cases where $\mathcal{N}_{0}(d; R)$ is a singleton. For $n=d$, (\ref{singleton2}) implies that the number of positive entries of the unique element in $\mathcal{N}_{0}(d; R)$ is exactly $d!$ unless that $r_{1} = \cdots = r_{d} =0$. For $n=d+1$, the number of positive entries of the unique element in $\mathcal{N}_{0}(d; R)$ is at most $(d+1)!$ by (\ref{singleton}). This gives another proof of the second half of Theorem \ref{Hakimi} when $\mathcal{N}_{0}(d; R)$ is a singleton.

Assume that $\mathcal{N}_{0}(d; R)$ is a nonempty convex polytope. In the following proposition, we show that an extreme point in $\mathcal{N}_{0}(d; R)$ must have at most $n \cdot d!$ positive entries. Since every convex polytope has an extreme point, the following proposition gives another proof of the second half of Theorem \ref{Hakimi} in case where $\mathcal{N}_{0}(d; R)$ is a convex polytope. The proposition is a generalization of the lemma in \cite{CK} and the proof is similar to the original one.

\begin{prop}
Let $d, n \in \mathbf{N}$ with $d \geq 2$ and $n \geq d$. Let $R = (r_{1}, \cdots, r_{n})$ be a vector with $r_{1} \geq \cdots \geq r_{n} \geq 0$. Assume that $\mathcal{N}_{0}(d; R)$ is a nonempty convex polytope. If $A \in \mathcal{N}_{0}(d; R)$ is an extreme point, then $A$ has at most $n \cdot d!$ positive entries.
\end{prop}

\begin{proof}
$A \in \mathcal{N}_{0}(d; R)$ corresponds to a solution of the linear system
\begin{equation}
\label{linear}
(d-1)!  \sum_{1 \leq i_{1} < \cdots <i_{d} \leq n} A(i_{1}, \ldots ,i_{d}) v_{i_{1}, \ldots, i_{d}} = R^{T}, \quad A(i_{1}, \ldots ,i_{d}) \geq 0,
\end{equation}
where each $v_{i_{1}, \ldots, i_{d}}$ is the vector in $\mathbf{R}^{n}$ that has 1's in $i_{1}$-th, $i_{2}$-th, $\ldots$ and $i_{d}$-th entries and 0's elsewhere. It is known that a solution of (\ref{linear}) is extreme if and only if the vectors $\{ v_{i_{1}, \ldots, i_{d}} \mid A(i_{1}, \cdots ,i_{d}) > 0\}$ are linearly independent (see \cite[Theorem 5 on p.\,55]{Gass}, for example). Hence the cardinality of $\{ A(i_{1}, \ldots ,i_{d}) > 0 \mid 1 \leq i_{1} < \cdots <i_{d} \leq n\}$ is at most $n$. Thus the claim follows from the symmetry of $A$.
\end{proof}

Finally, we rewrite Theorem \ref{Hakimi} in terms of hypergraph theory. Let $V = \{v_{1}, \ldots, v_{n} \}$ be a finite set and $d \in \mathbf{N}$ with $d \geq 2$. Let $e_{1}, \ldots ,e_{m}$ be subsets of $V$ such that $\cup_{i=1}^{m}e_{m} = V$ and $|e_{i}| = d$ for each $i \in \langle m \rangle$. Set $E = \{e_{i} \mid i \in \langle m \rangle \}$. Let $w : E \rightarrow \mathbf{R}$ be a map and we assume that $w \geq 0$ in what follows. Then the triple $(V, E, w)$ is called a \textit{$d$-uniform weighted hypergraph}. The elements of $V$, $E$ and $\{ w(e_{i}) \in \mathbf{R} \mid i \in \langle m \rangle \}$ are called \textit{vertices}, \textit{edges} and \textit{weights} respectively. For each $i \in \langle n \rangle$, we set
\begin{equation*}
d_{i} := \sum_{v_{i} \in e_{j}} w(e_{j})
\end{equation*}
and call it the \textit{weighted degree at $v_{i}$}. The sequence $D := (d_{1}, \cdots, d_{n})$ is called the \textit{weighted degree sequence} of $(V, E, w)$. For a given strongly hollow symmetric nonnegative $d$-hypermatrix $A$, one can associate a $d$-uniform weighted hypergraph $(V, E, w)$ as follows: Set 
\begin{equation*}
V = \{v_{1}, \ldots, v_{n} \},
\end{equation*}
\begin{equation*}
E = \left\{  e_{i_{1}\cdots i_{d}} := \{v_{i_{1}}, \ldots, v_{i_{d}}\} \mid 1 \leq i_{1} < \cdots < i_{d} \leq n \right\},
\end{equation*}
and
\begin{equation*}
w(e_{i_{1}\cdots i_{d}}) = A(i_{1}, \ldots, i_{d}).
\end{equation*}
Then the triple $(V, E, w)$ is a $d$-uniform weighted hypergraph. Conversely, for a given $d$-uniform weighted hypergraph $(V, E, w)$, one can associate a strongly hollow symmetric nonnegative $d$-hypermatrix. Consider a strongly hollow symmetric nonnegative $d$-hypermatrix $A$ whose $j$-th slice sum is $r_{j} \geq 0$. Then the associated $d$-uniform weighted hypergraph $(V, E, w)$ satisfies
\begin{equation*}
d_{j} = \frac{r_{j}}{(d-1)!}
\end{equation*}
for each $j \in \langle n \rangle$. Thus Theorem \ref{Hakimi} can be interpreted as a result about realizability of a $d$-uniform weighted hypergraph for a given weighted degree sequence. 

\begin{theo}
Let $d, n \in \mathbf{N}$ with $d \geq 2$ and $n \geq d$. Let $D = (d_{1}, \ldots, d_{n})$ be a sequence with $d_{1} \geq \cdots \geq d_{n} \geq 0$. Then there exists a $d$-uniform weighted hypergraph with the weighted degree sequence $D$ if and only if the inequality
\begin{equation}
\label{condition2}
dd_{1} \leq \sum_{i=1}^{n}d_{i}
\end{equation}
holds. Moreover, if this holds, then there exists a $d$-uniform weighted hypergraph with the weighted degree sequence $D$ that has at most $n$ edges with positive weights.
\end{theo} 

For $n=d$ or $d+1$, Lemma \ref{polytope} and the above theorem imply the following corollary:
\begin{corr}
Let $d \in \mathbf{N}$ with $d \geq 2$ and $n =d$ or $d+1$. Let $D = (d_{1}, \ldots, d_{n})$ be a sequence with $d_{1} \geq \cdots \geq d_{n} \geq 0$. If the inequality $(\ref{condition2})$ holds, then there exists a unique $d$-uniform weighted hypergraph with the weighted degree sequence $D$.
\end{corr}
A $d$-uniform hypergraph is called \textit{simple} if it has no repeated edges. The degree sequence of a simple $d$-uniform hypergraph $(V,E)$ can be seen as the weighted degree sequence of a $d$-uniform weighted hypermatrix $(V, E, w)$ with $w(e) =1$ for any $e \in E$.  A nonincreasing sequence of nonnegative integers is called \textit{$d$-graphic} if it can be realized as the degree sequence of some simple $d$-uniform hypergraph. Dewdney \cite{Dewdney} proved the following:
\begin{theo}\emph{(\cite{Dewdney})}
\label{DewdneyCriterion}
Let $d, n \in \mathbf{N}$ with $d \geq 2$ and $n \geq d$. Let $D= (d_{1}, \ldots, d_{n})$ be a nonincreasing sequence of nonnegative integers. Then $D$ is $d$-graphic if and only if there exists a nonincreasing sequence $D' = (d'_{2}, \ldots, d'_{n})$ of nonnegative integers satisfying all of the following conditions:
\begin{enumerate}
\item $D'$ is $(d-1)$-graphic. 
\item $\sum_{i=2}^{n} d'_{i} = (d-1)d_{1}$.
\item $D'' := (d_{2}-d'_{2}, \ldots, d_{n}-d'_{n}, 0)$ is $d$-graphic.
\end{enumerate}
\end{theo}
\begin{rema}
If $D= (d_{1}, \ldots, d_{n})$, a nonincreasing sequence of nonnegative integers, is $d$-graphic, then $D$ satisfies the condition $(\ref{condition2})$. In fact, if $D$ is $d$-graphic, there exists $D'$ as in Theorem \ref{DewdneyCriterion}. Then we have
\begin{equation*}
dd_{1} = d_{1} + \sum_{i=2}^{n} d'_{i} \leq d_{1} + \sum_{i=2}^{n} d_{i} = \sum_{i=1}^{n}d_{i},
\end{equation*}
where the condition $(2)$ is used at the left equality and the middle inequality follows from the condition $(3)$. This remark is reasonable since the degree sequence of a simple $d$-uniform hypergraph $(V,E)$ is nothing but the weighted degree sequence of a $d$-uniform weighted hypergraph $(V, E, w)$ with $w(e) =1$ for any $e \in E$. 
\end{rema}

\section*{Acknowledgments.} I am indebted to Takefumi Kondo for giving me advice to use Rado's theorem \cite{Rado} for finding the convex hulls of Grassmanians. Before his advice, I had used Hakimi's theorem \cite{Hakimi} and I could obtain the results about only Grassmannians of two dimensional subspaces. This research was largely expanded by his advice. After I wrote the first version of this paper, Josse van Dobben de Bruyn made me realize that Fillmore and Williams \cite{FW} obtained the convex hulls of Grassmannians by a different method. I appreciate his comment. I also thank my advisor Shin Nayatani for valuable discussions and comments. Finally, I would like to thank the anonymous referees, whose careful proofreading and valuable comments have greatly improved the paper.


\begin{thebibliography}{99}
\bibitem{AAS} N. Achuthan, N. Achuthan and M. Simanihuruk, On 3-uniform hypergraphic sequences, \textit{J. Combin. Math. Combin. Comput.} \textbf{14} (1993), 3--13. 
\bibitem{Behrens} S. Behrens, C. Erbes, M. Ferrara, S. G. Hartke, B. Reiniger, H. Spinoza and C. Tomlinson, New results on degree sequences of uniform hypergraphs, \textit{Electron. J. Combin.} \textbf{20} (2013), \#P14.
\bibitem{Brualdi} R. A. Brualdi, \textit{Combinatorial Matrix Classes}, Cambridge University Press, 2006. 
\bibitem{BD} R. A. Brualdi and G. Dahl, Alternating sign matrices and hypermatrices, and a generalization of Latin squares, \textit{Adv. Appl. Math}. \textbf{95} (2018), 116--151.
\bibitem{CK} G. Converse and M. Katz, Symmetric matrices with given row sums, \textit{J. Combin. Theory} \textbf{18} (1975), 171--176.
\bibitem{Dewdney} A. K. Dewdney, Degree sequences in complexes and hypergraphs, \textit{Proc. Amer. Math. Soc.} \textbf{53} (1975), 535--540.
\bibitem{Dimitric} I. Dimitri\'{c}, A Note on Equivariant Embeddings of Grassmannians, \textit{Publ. Inst. Math. (Beograd)} \textbf{59} (\textbf{73}) (1996), 131--137.
\bibitem{EG} P. Erd\H{o}s and T. Gallai, Graphs with prescribed degrees of vertices (Hungarian), \textit{Mat. Lapok}. \textbf{11} (1960), 264--274.
\bibitem{Fan}Ky Fan, On a theorem of Weyl concerning the eigenvalues of linear transformations, \textit{Proc. Nat. Acad. Sci. U.S.A.} \textbf{35} (1949), 652--655.  
\bibitem{FW} P. A. Fillmore and J. P. Williams, Some convexity theorems for matrices, \textit{Glasgow Math. J.} \textbf{12} (1971), 110--117.
\bibitem{Frosini} A. Frosini, C. Picouleau and S. Rinaldi, New sufficient conditions on the degree sequences of uniform hypergraphs, \textit{Theoret. Comput. Sci.} \textbf{868} (2021), 97--111.
\bibitem{Gass} S. I. Gass, \textit{Linear Programming. Methods and Applications, 4th edition.} McGraw-Hill, New York, 1975.
\bibitem{Hakimi} S. L. Hakimi, On realizability of a set of integers as degrees of the vertices of a linear graph I, \textit{J. Soc. Indust. Appl. Math.} \textbf{10} (1962), 496--506.
\bibitem{Halmos} P. R. Halmos, \textit{A Hilbert Space Problem Book}, Van Norstrand, Princeton, 1967.
\bibitem{Havel} V. Havel, Eine Bemerkung \"{u}ber die Existenz der endlichen Graphen (Czech), \textit{\v{C}asopis P\v{e}st. Mat.} \textbf{80} (1955), 477--480. 
\bibitem{HS} H. Hoogeveen and G. Sierksma, Seven criteria for integer sequences being graphic, \textit{J. Graph Theory} \textbf{15} (1991), 223--231.
\bibitem{MO} A. W. Marshall, I. Olkin and B. C. Arnold, \textit{Inequalities: Theory of Majorization and Its Applications, Second Edition, Springer Series in Statistics}, Springer, New York, 2011.
\bibitem{OW} M. L. Overton and R. S. Womersley, On the sum of the largest eigenvalues of a symmetric matrix, \textit{SIAM J. Matrix Anal. Appl.} \textbf{13} (1992), no.1, 41--45.
\bibitem{Rado} R. Rado, An inequality, \textit{J. London Math. Soc.} \textbf{27} (1952), 1--6. 
\bibitem{Tai} S. S. Tai, Minimum imbeddings of compact symmetric spaces of rank one, \textit{J. Diff. Geom.} \textbf{2} (1968), 55--66.
\bibitem{Zhan} X. Zhan, The sharp Rado theorem for majorizations, \textit{Amer. Math. Monthly} \textbf{110} (2003), no. 2, 152--153.
\bibitem{Zhan2} X. Zhan, \textit{Matrix Theory}, Graduate Studies in Mathematics, \textbf{147}, Amer. Math. Soc., Providence, RI, 2013.
\end{thebibliography}
\end{document}